\documentclass[12pt,showkeys]{amsart}
\usepackage{color}
\usepackage{graphicx}
\usepackage{subfigure}
\usepackage{amssymb}
\usepackage{amsmath}
\usepackage{multirow}
\usepackage{enumerate}
\usepackage{epstopdf}
\usepackage{cite,stmaryrd,txfonts}
\usepackage{tikz}

\input{undertilde}
\usetikzlibrary{snakes}

\usepackage{verbatim}

\usepackage{epic}

\usepackage{empheq}

\usepackage{cases}

\newcommand{\bs}{\boldsymbol}
\newcommand{\mrm}{\mathrm}

\def\cequiv{\raisebox{-1.5mm}{$\;\stackrel{\raisebox{-3.9mm}{=}}{{\sim}}\;$}}

\def\r{\mathbf{r}}
\def\s{\mathbf{s}}

\newtheorem{theorem}{Theorem}
\newtheorem{remark}[theorem]{Remark}

\newtheorem{lemma}[theorem]{Lemma}

\newcounter{mnote}
\let\oldmarginpar\marginpar
\renewcommand\marginpar[1]{\-\oldmarginpar[\raggedleft\footnotesize #1]%
  {\raggedright\footnotesize #1}}

\def\r{\mathbf{r}}
\def\s{\mathbf{s}}
\def\dx{\mathrm{dx}}
\def\dS{\,\mathrm{ds}\,}
\def\q{\mathcal{Q}}
\def\m{\mathfrak{mix}}

\begin{document}

\title[Stable Stokes finite element pair on quadrilateral grids]{Stable finite element pair for Stokes problem and discrete Stokes complex on quadrilateral grids}


\author{
      Shuo Zhang 
}

\address{LSEC, Institute of Computational Mathematics, Academy of Mathematics and System Sciences, Chinese Academy of Sciences, Beijing 100190, China.}
\email{szhang@lsec.cc.ac.cn}
\thanks{The author is partially supported by NSFC11101415 and  National Center for Mathematics and Interdisciplinary Sciences(NCMIS).}

\keywords{incompressible Stokes problem, quadrilateral grid, stable finite element pair, Stokes complex}
\subjclass[2010]{65M60, 76M10}


\maketitle

\begin{abstract}
{
In this paper, we first construct a nonconforming finite element pair for incompressible Stokes problem on quadrilateral grids, and then construct a discrete Stokes complex associated with that finite element pair. The finite element spaces involved consist of piecewise polynomials only, and the divergence-free condition is imposed in a primal formulation. Combined with some existing results, these constructions can be generated onto grids that consist of both triangular and quadrilateral cells. 
}


\end{abstract}


\section{Introduction}

The Stokes problem is an important model problem in applied sciences, which can be used to describe the motion of an incompressible fluid. In this paper, we study the stable finite element method for the two dimensional stationary incompressible Stokes problem of velocity-pressure type. In this context, a stable finite element method of the model problem implies a pair of finite element spaces that are consistent approximations to the Hilbert spaces $(H^1(\Omega))^2$ and $L^2(\Omega)$, respectively, and satisfy the two stability conditions which are, with the detailed technical 
description given in the following section, 

~~~~\textbf{SC 1}: the coercive condition and inf-sup condition hold uniformly;

~~~~\textbf{SC 2}: the divergence-free constraint is imposed in a primal formulation.
\\
The first condition falls into the classical theory of Stokes problem, as it provides a necessary and sufficient condition for the well-posedness of the discrete problem; see, e.g., \cite{Girault.V;Raviart.P,Babuska.I1973,Brezzi1974}. The second condition, sometimes known as mass conservation, is also desirable in many applications. Though it is not yet fully revealed how methods that enforce divergence-free would be superior to those that do not, satisfying the property can decouple the pressure error from the velocity error, and can avoid possible instabilities that can arise from violation of mass conservation \cite{Auricchio.F;daVeiga.L;Lovadina.C;Reali.A2010,Cousins.B;Borne.S;Linke.A;Rebholz.L;Wang.Z2013,Linke.A2009,Linke.A;Matthies.G;Tobiska.L2008}. Various approaches have been utilised to develop methods in regard to the conditions both \textbf{SC 1} and \textbf{SC 2}, including the discontinuous Galerkin methods \cite{Carrero.J;Cockburn.B;Schotzau.D2006,Cockburn.B;Kanschat.G;Schotzau.D2007,Cockburn.B;Kanschat.G;Schotzau.D2004,Cockburn.B;Kanschat.G;Schotzau.D2005}, the isogeometric methods \cite{Evans.J;Hughes.T2012,Evans.J;Hughes.T2013}, the least square finite element methods \cite{Bochev.P;Lai.J;Olson.L2013,Bolton.P;Thatcher.R2005,Chang.C;Nelson.J1997,Heys.J;Lee.E;Manteuffel.T;McCormick.S2006}, and finite element methods with enhanced stabilisation\cite{Bejanov.B;Guermond.J;Minev.P2005,Boffi.D;Cavallini.N;Gardini.F;Gastaldi.L2012,Burman.E;Linke.A2008,Case.M;Ervin.V;Linke.A;Rebholz.L2011,Olshanskii.M;Reusken.A2004}. In this paper, we focus ourselves on the traditional finite element methods, and will not mention other approaches too much.

A few finite element pairs have been reported to satisfy the conditions both \textbf{SC 1} and \textbf{SC 2}. As a natural idea, the conforming $P_k^2-P_{k-1}$ pairs were constructed for $k\geqslant2$. They are proved to satisfy \textbf{SC 1} and \textbf{SC 2} on special types of uniform or quasi-uniform triangulations (Scott-Vogelius\cite{Scott.L;Vogelius.M1985,Scott.L;Vogelius.M1985LAM} for $k\geqslant 4$, Arnold-Qin\cite{Arnold.D;Qin.J1992} for $k=2$, and Qin \cite{Qin.J1994} for $k=3$). By adding extra smoothness to the finite element functions on the vertices other than the corners of the domain, Falk-Neilan \cite{Falk.R;Neilan.M2013} designed a special family of $P_k^2-P_{k-1}$ pairs for $k\geqslant 4$ that are shown to satisfy \textbf{SC 1} and \textbf{SC 2} on general grids without the so-called singular corner vertices. On general triangular grids, Crouzeix-Raviart \cite{Crouzeix.M;Raviart.P1973} constructed a nonconforming $P_1^2-P_0$ element which satisfies both \textbf{SC 1} and \textbf{SC 2} in a nonconforming way. A similar nonconforming $P_2^2-P_1$ element was constructed by Fortin-Soulie\cite{Fortin.M;Soulie.M1983}. Another natural idea is, for a given velocity space, using its divergence space as the pressure space, and/or reversely, for a given pressure space, looking for a velocity space so that divergence is a surjection. This idea succeeds where nodal basis functions can be constructed, such as the $Q_{k+1,k}\times  Q_{k,k+1}-Q_{k}$ pair on rectangular grid (see Zhang \cite{Zhang.S2009} for $k\geqslant 2$ and Huang-Zhang \cite{Huang.Y;Zhang.S2011} for $k=1$), and the Mardal-Tai-Winther pair\cite{Mardal.K;Tai.X;Winther.R2002} on triangular grids, which uses a space of vector functions rather than a tensor product of two scalar function spaces to approximate the velocity field. Kouhia-Stenberg \cite{Kouhia.K;Stenberg.R1995} also used different function spaces for different components of the velocity and constructs a stable linear element method. Xie-Xu-Xue\cite{Xie.X;Xu.J;Xue.G2008} made a way to add divergence-free basis functions onto $H(\mathrm{div})$-conforming finite element space to generate the velocity function space, and generate and survey several stable pairs in a unified way. Guzm\'an-Neilan \cite{Guzman.J;Neilan.M} constructed the velocity spaces by adding \emph{rational} divergence-free functions to $H(\mathrm{div})$-conforming functions, and obtained conforming stable pairs. Beside these examples, there have been many finite element pairs that satisfy the condition \textbf{SC 1}, while satisfy the divergence-free condition \textbf{SC 2} in a dual formulation; for these pairs, see \cite{Boffi.D;Brezzi.F;Fortin.M2008,Brezzi.F;Fortin.M1991,Girault.V;Raviart.P,Rannacher.R2000} and the references therein. 

Principally, the condition \textbf{SC 2} decouples the computations of pressure and velocity, and it would bring convenience once we can present a precise description of the divergence-free velocity subspace. Mathematically, a divergence-free function can be the curl of some other function\cite{Girault.V;Raviart.P}. This is relevant to the fundamental observation that the incompressible velocity field admits a stream function, and this gives a natural connection between the incompressible Stokes problem and the biharmonic equation. This property can be described in the framework of the Stokes complex originally introduced by \cite{Mardal.K;Tai.X;Winther.R2002,Tai.X;Winther.R2006}. There have been various complexes to describe different physical and mathematical observations\cite{Arnold.D;Falk.R;Winther.R2006m,Arnold.D;Falk.R;Winther.R2006}. A powerful tool to design, analyse, understand, and moreover, to apply the stable finite element pair for the Stokes problem is then to reproduce discrete analogous of the Stokes complex where the Sobolev spaces are replaced by corresponding finite element spaces. The existence of such a structure like the discrete Stokes complex makes the connection between the model problems revealed at discrete level, and wider scope of methods and applications of the model problems can be expected. This structure is an intrinsic connection between the finite element pairs, and some of the existing pairs that satisfy both \textbf{SC 1} and \textbf{SC 2} have been shown to be associated with specific discrete Stokes complexes. On triangular grids, for instance, discrete Stokes complexes have been established associated with the conforming $P_k^2-P_{k-1}$ element, with(\cite{Guzman.J;Neilan.M}) or without(\cite{Scott.L;Vogelius.M1985}) extra smoothness on vertices, and the nonconforming $P_1^2-P_0$ element pair \cite{Falk.R;Morley.E1990}, respectively. Different discrete Stokes complexes were also introduced in \cite{Mardal.K;Tai.X;Winther.R2002} and \cite{Guzman.J;Neilan.M}, respectively. While when quadrilateral grids are considered, few discrete Stokes complexes are known.


As the quadrilateral grids are widely used where the problem geometry is of quadrilateral nature, in this paper, we study the stable finite element method for Stokes problem that satisfy conditions both \textbf{SC 1} and \textbf{SC 2} on quadrilateral grids. Specifically, as it is seen among the existing methods that the nonconforming methodology would in general admit higher flexibility, we develop a nonconforming finite element pair that satisfies both \textbf{SC 1} and \textbf{SC 2} in a nonconforming way, and then construct a discrete Stokes complex associated with this element pair.  After carry out the discussion on quadrilateral grids, we will then carry out the discussion on grids consisting of both triangular and rectangular cells to obtain parallel analogue results.

On the quadrilateral grids, we use an average continuous piecewise incomplete quadratic polynomial space for the velocity field. The velocity space is the same as the one  used for solving the Poisson equation in Lin-Tobiska-Zhou \cite{Lin.Q;Tobiska.L;Zhou.A2005} on rectangular grid, while the nodal parameters of the velocity space had been used by Han \cite{Han.H1984} with different shape function space on rectangle grids. The same velocity space was also used in Shi-Zhang \cite{Shi.D;Zhang.Y2006} for rectangular grids to shape a finite element pair for the Stokes problem, where the pressure is approximated by piecewise constant, and the condition \textbf{SC 2} holds in a dual formulation. Shi-Zhang's element also relies on a bilinear mapping between the cell and a reference rectangle when forming the shape functions on quadrilateral cells. In this present paper, applying the idea in \cite{Park.C;Sheen.D2013}, we define the finite element functions on quadrilateral cells directly, and they are all piecewise polynomials. We use discontinuous piecewise linear polynomial space for the pressure, and both the conditions \textbf{SC 1} and \textbf{SC 2} are satisfied.

A discrete Stokes complex is then constructed based on the newly established finite element pair as we prove that the divergence-free part of the discrete velocity space is piecewisely the curl of a quadrilateral Morley element space. The Morley element was originally constructed on triangular grids to solve the fourth order problem with piecewise quadratic polynomials\cite{Morley.L1968}, and generalized to arbitrary dimension by Wang-Xu \cite{Wang.M;Xu.J2006}, and to elliptic problems of arbitrary order by Wang-Xu \cite{WangXu2012}. Using the same nodal parameters as the Morley element, Wang-Shi-Xu\cite{Wang.M;Shi.Z;Xu.J2007} constructs a rectangle Morley element, which is generalised by Park-Sheen \cite{Park.C;Sheen.D2013} to general convex quadrilateral grids. The Morley element was proved to be associated with a discrete Stokes complex on triangles together with the nonconforming $P_1^2-P_0$ element \cite{Falk.R;Morley.E1990,Feng.C;Xu.J;Zhang.S2013}.  In this present paper, we show that a discrete Stokes complex connects the quadrilateral Morley element and the newly-developed Stokes element pair on quadrilateral grids. 

Because of the similarity of the nodal parameters of the finite elements established on triangular and quadrilateral cells, which implies the same continuity of finite element functions on triangular and quadrilateral triangulations, it is then natural to combine these finite elements together to form discretisation schemes for the biharmonic problem and the Stokes problem, respectively, on a mixed grid involving both triangular and quadrilateral cells. Since the nodal interpolations are defined locally, this combination is straightforward, and finally a same discrete Stokes complex is also established on the mixed grid.

%
The rest of the paper is as follows. In Section \ref{sec:pre}, we introduce some preliminaries including the model problems and general finite element discretisation. In Section \ref{sec:stap}, we introduce a stable finite element pair on quadrilateral grids, and construct a discrete Stokes complex. In Section \ref{sec:mixgrid}, we carry out the discussion on a mixed grid. Finally, conclusions are given in Section \ref{sec:con}. 


%
%
\section{Preliminaries}
\label{sec:pre}

\subsection{Stokes problem and the Stokes complex}

Let $\Omega\subset\mathbb{R}^2$ be a Lipschitz domain, and $\Gamma=\partial\Omega$ be the boundary, with $\mathbf{n}$ the outward unit normal vector. We consider the incompressible Stokes problem with homogeneous boundary condition:
\begin{equation}\label{eq:stokesbvp}
\left\{\begin{array}{rl}
-\nu\Delta\mathbf{u}+\nabla p=\mathbf{f} & \mbox{in}\,\Omega, \\
\nabla\cdot\mathbf{u}=0 &\mbox{in}\,\Omega, \\
\mathbf{u}=0 &\mbox{on}\,\partial\Omega.
\end{array}
\right.
\end{equation}
Here $\nu$ is the kinematic viscosity, $\mathbf{u}$, $p$, and $\mathbf{f}$ denote the velocity, the pressure, and the external body force, respectively, and $\Delta$ and $\nabla$ are the Laplacian and gradient operators, respectively. For simplicity, we set $\nu=1$ in the rest of the paper.  

Denote by $H^1(\Omega)$, $H^1_0(\Omega)$, $H^2(\Omega)$, and $H^2_0(\Omega)$ the standard Sobolev spaces as usual, and $L^2_0(\Omega):=\{w\in L^2(\Omega):\int_\Omega w\dx=0\}$. The variational form of \eqref{eq:stokesbvp} is to find $(\mathbf{u},p)\in (H^1_0(\Omega))^2\times L^2_0(\Omega)$, such that 
\begin{equation}\label{eq:stokesvp}
\left\{
\begin{array}{rll}
(\nabla\mathbf{u},\nabla\mathbf{v})+(\nabla\cdot\mathbf{v},p)&=(\mathbf{f},\mathbf{v}) & \forall\,\mathbf{v}\in (H^1_0(\Omega))^2, \\ 
(\nabla\cdot\mathbf{u},q)&=0 & \forall\,q\in L^2_0(\Omega).
\end{array}
\right.
\end{equation}
Here $(\nabla \mathbf{u},\nabla\mathbf{v})=\int_\Omega\sum_{i,j=1}^2(\nabla\mathbf{u})_{ij}(\nabla\mathbf{u})_{ij}\dx$, and $(\mathbf{f},\mathbf{v})=\int_\Omega\sum_{i=1}^2\mathbf{f}_i\mathbf{v}_i\dx$. The well-posedness of \eqref{eq:stokesvp} is guaranteed by the facts\cite{Girault.V;Raviart.P}:
\begin{equation}\label{eq:contsta}
(\nabla\mathbf{v},\nabla\mathbf{v})\geqslant C_1 \|\mathbf{v}\|_{1,\Omega}^2\ \ \forall\,\mathbf{v}\in (H^1_0(\Omega))^2,\ \ \mbox{and}\ \ \inf_{0\neq q\in L^2_0(\Omega)}\sup_{\mathbf{0}\neq\mathbf{v}\in (H^1_0(\Omega))^2}\frac{(\nabla\cdot\mathbf{v},q)}{\|\mathbf{v}\|_{1,\Omega}\|q\|_{0,\Omega}}\geqslant C_2,
\end{equation}
with $C_1$ and $C_2$ two positive constants dependent on the domain only. 


The biharmonic problem associated with the Stokes problem \eqref{eq:stokesbvp} is:
\begin{equation}\label{eq:bihbvp}
\left\{
\begin{array}{ll}
\displaystyle\Delta^2\varphi=F\in H^{-2}(\Omega),&\mbox{in}\,\Omega, \\
\displaystyle\varphi=\frac{\partial\varphi}{\partial\mathbf{n}}=0, &\mbox{on}\,\partial\Omega.
\end{array}
\right.
\end{equation}
The variational problem is to find $\varphi\in H^2_0(\Omega)$, such that 
\begin{equation}
(\nabla^2\varphi,\nabla^2\psi)=F(\psi),\ \ \forall\,\psi\in H^2_0(\Omega).
\end{equation}
Here $(\nabla^2\varphi,\nabla^2\psi)=\int_\Omega\sum_{i,j=1}^2(\nabla^2\varphi)_{ij}(\nabla^2\psi)_{ij}\dx$.
Denote by $\mathbf{curl}$ the curl operator on a scalar function, which is the rotation of the gradient operator $\nabla$. Define for vector functions the operators $\mathrm{div}=\nabla\cdot$ and $\mathrm{curl}=\mathbf{curl}\cdot$.  We have the basic relation below.
\begin{lemma}\label{lem:conbihsto}\cite{Girault.V;Raviart.P}
Let $\Omega$ be simply connected.
\begin{enumerate}
\item $\mathbf{curl} H^2_0(\Omega)=\{\mathbf{v}\in (H^1_0(\Omega))^2:\mathrm{div}\mathbf{v}=0\}$, and $\mathrm{div} (H^1_0(\Omega))^2=L^2_0(\Omega). $
\item Let $(\mathbf{u},p)$ be the solution of \eqref{eq:stokesvp}, and $\varphi$ be the solution of \eqref{eq:bihbvp}, with $F(\psi)=(\mathbf{f},\mathbf{curl}\psi)$ for $\psi\in H^2_0(\Omega)$. Then $\mathbf{u}=\mathbf{curl}\varphi$.
\end{enumerate}
\end{lemma}

We can rewrite Lemma \ref{lem:conbihsto} in the form of the Stokes complex \cite{Mardal.K;Tai.X;Winther.R2002,Tai.X;Winther.R2006,Falk.R;Neilan.M2013} which reads
\begin{equation}
\begin{array}{ccccccccc}
0 & ~~~\longrightarrow~~~ & H^2_0(\Omega) & ~~~\xrightarrow{\bs{\mrm{curl}}}~~~ & (H^1_0(\Omega))^2 & ~~~\xrightarrow{\mrm{div}}~~~ & L^2_0(\Omega)  & ~~~\rightarrow~~~ & 0.  
\end{array}
\end{equation}
In this sequence, the composition of two consecutive
mappings is zero, and the range of each map is the null space of the succeeding map in a simply connected domain. 

\subsection{Finite element method for Stoke problem}

When the Sobolev spaces $(H^1_0(\Omega))^2$ and $L^2_0(\Omega)$ are replaced by some finite element spaces $\mathbf{V}_{h0}$ and $\mathring{W}_h$, conforming or nonconforming, we have the finite element problem: find $(\mathbf{u}_h,p_h)\in \mathbf{V}_{h0}\times \mathring{W}_h$, such that 
\begin{equation}\label{eq:stokesfep}
\left\{
\begin{array}{rll}
(\nabla_h\mathbf{u}_h,\nabla\mathbf{v}_h)+(\mathrm{div}_h\mathbf{v}_h,p_h)&=(\mathbf{f},\mathbf{v}_h) & \forall\,\mathbf{v}_h\in \mathbf{V}_{h0}, \\ 
(\mathrm{div}_h\mathbf{u}_h,q_h)&=0 & \forall\,q_h\in \mathring{W}_h.
\end{array}
\right.
\end{equation}
Here $\nabla_h$ and $\mathrm{div}_h$ are in the piecewise sense for nonconforming $V_{h0}$.

We define two stable conditions for the finite element scheme. The subscript $``h"$ in the norms implies the dependence of the triangulation.

~~~~\textbf{SC 1}: There exist two positive constants $\gamma_1$ and $\gamma_2$, such that
$$
\displaystyle\inf_{\mathbf{v}_h\in\mathbf{Z}_h,\mathbf{v}_h\neq0}\frac{(\nabla_h\mathbf{v}_h,\nabla_h\mathbf{v}_h)}{\|\mathbf{v}_h\|_{1,h}^2}:=\gamma_h^1>\gamma_1\ \mbox{on}\ \mathbf{Z}_h:=\{\mathbf{v}_h\in \mathbf{V}_{h0}:(\mathrm{div}_h\mathbf{v}_h,q_h)=0,\ \forall\,q_h\in \mathring{W}_h\}
$$ 
and $\displaystyle\inf_{0\neq q_h\in \mathring{W}_h}\sup_{\mathbf{0}\neq\mathbf{v}_h\in\mathbf{V}_{h0}}\frac{(\mathrm{div}_h\mathbf{v}_h,q_h)}{\|\mathbf{v}_h\|_{1,h}\|q_h\|_0}:=\gamma_h^2>\gamma_2$. 

~~~~\textbf{SC 2}: $\mathbf{Z}_h=\{\mathbf{v}_h\in\mathbf{V}_{h0}:\mathrm{div}_h\mathbf{v}_h=0\}$.

Evidently, a necessary condition that both of the two conditions hold is that the pressure space is the divergence space of the velocity space, in a conforming or nonconforming way.  

We have the convergence result for the finite element problem.
\begin{lemma}\label{lem:clm}\cite{Arnold.D1994,Brezzi.F;Fortin.M1991}
Let the stable condition \textbf{SC 1} hold. Then the discrete problem \eqref{eq:stokesfep} has a unique solution. Moreover, let $(\mathbf{u},p)$ and $(\mathbf{u}_h,p_h)$ be the solution of \eqref{eq:stokesvp} and \eqref{eq:stokesfep}, respectively, then there exists a constant $C$ depending only on $\gamma_1$ and $\gamma_2$, such that
\begin{multline*}
\|\mathbf{u}-\mathbf{u}_h\|_{1,h}+\|p-p_h\|_{0,\Omega}\leqslant C\Bigg(\inf_{\mathbf{v}_h\in\mathbf{V}_{h0}}\|\mathbf{u}-\mathbf{v}_h\|_{1,h}+\inf_{q_h\in \mathring{W}_h}\|p-q_h\|_0 \\ 
+\sup_{0\neq\mathbf{v}_h\in\mathbf{V}_h}\frac{(\nabla_h\mathbf{u},\nabla_h\mathbf{v}_h)+(\mathrm{div}_h\mathbf{v}_h,p)-(\mathbf{f},\mathbf{v}_h)}{\|\mathbf{v}\|_{1,h}}\Bigg).
\end{multline*}
\end{lemma}
The last term is the consistency error, which vanishes when $\mathbf{V}_{h0}\subset (H^1_0(\Omega))^2$.

%
%

\section{Stable finite element pair and discrete Stokes complex on quadrilateral grids}
\label{sec:stap}

\subsection{Quadrilateral triangulation}

\subsubsection{Geometry of convex quadrilateral grid}

Let $Q$ be a convex quadrilateral with $a_i$ the vertices and $e_i$ the edges, $i=1:4$. See Figure \ref{fig:convquad} for an illustration. Let $m_i$ be the mid-point of $e_i$, then the quadrilateral $\square m_1m_2m_3m_4$ is a parallelogram(\cite{Park.C;Sheen.D2003}). The cross point of $m_1m_3$ and $m_2m_4$, which is labelled as $O$, is the midpoint of both $m_1m_3$ and $m_2m_4$. Denote $\mathbf{r}=\overrightarrow{Om_3}$ and $\mathbf{s}=\overrightarrow{Om_4}$. Then the coordinates of the vertices in the coordinate system $\r O\s$ are $a_1(-1-\alpha,-1-\beta)$, $a_2(-1+\alpha,-1+\beta)$, $a_3(-1+\alpha,-1+\beta)$ and $a_4(1+\alpha,1+\beta)$ for some $\alpha,\beta$. 
Since $Q$ is convex, $|\alpha|+|\beta|<1$(\cite{Park.C;Sheen.D2013}). Without loss of generality, we assume $\alpha>0$, $\beta>0$ and $\r\times\s>0$.

Define the shape regularity indicator of the of the cell $Q$ by $\mathcal{R}_Q:=\max\{\frac{|\mathbf{r}||\mathbf{s}|}{\mathbf{r}\times\mathbf{s}},\frac{|\mathbf{r}|}{|\mathbf{s}|},\frac{|\mathbf{s}|}{|\mathbf{r}|}\}$. Evidently $\mathcal{R}_Q\geqslant 1$, and $\mathcal{R}_Q=1$ if and only if $Q$ is a square. A given family of quadrilateral triangulations $\{\mathcal{Q}_h\}$ of $\Omega$ is said to be regular, if all the shape regularity indicators of the cells of all the triangulations are uniformly bounded.

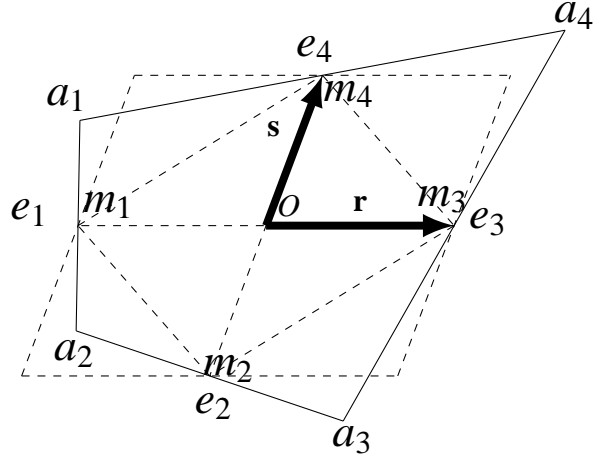
\begin{figure}[htbp]
\begin{tikzpicture}
\draw (0.05,0.26)node{$O$};

\draw[dashed](-2,2)--(3,2);
\draw[dashed](-3.5,-2)--(1.5,-2);
\draw[dashed](-3.5,-2)--(-2,2);
\draw[dashed](1.5,-2)--(3,2);

\draw[dashed](0.5,2)--(-1,-2);
\draw[dashed](2.25,0)--(-2.75,0);

\draw(-2.345,0.275)node{\Large$m_1$};
\draw(-0.75,-1.9)node{\Large$m_2$};
\draw(2.1,0.35)node{\Large$m_3$};
\draw(0.85,1.75)node{\Large$m_4$};

\draw(-3.4,0.175)node{\Large$e_1$};
\draw(-0.95,-2.4)node{\Large$e_2$};
\draw(2.7,0.05)node{\Large$e_3$};
\draw(0.4,2.35)node{\Large$e_4$};

\draw[->,line width=3pt,-latex](-0.25,0)-- node[auto] {$\r$} (2.25,0);
\draw[->,line width=3pt,-latex](-0.25,0)-- node[auto] {$\s$} (0.5,2);

\draw(3.725,2.6)--(-2.725,1.4);
\draw(3.725,2.6)--(0.775,-2.6);
\draw(-2.775,-1.4)--(0.775,-2.6);
\draw(-2.775,-1.4)-- (-2.725,1.4);

\draw(-2.9,1.7)node{\Large$a_1$};
\draw(-2.8,-1.7)node{\Large$a_2$};
\draw(0.9,-2.85)node{\Large$a_3$};
\draw(3.85,2.8)node{\Large$a_4$};

\draw[dashed](0.5,2)--(2.25,0);
\draw[dashed](2.25,0)--(-1,-2);
\draw[dashed](-1,-2)--(-2.75,0);
\draw[dashed](-2.75,0)--(0.5,2);

\end{tikzpicture}

\caption{Illustration of a convex quadrilateral $Q$.}\label{fig:convquad}
\end{figure}

Define two linear functions $\xi$ and $\eta$ by $\xi(a\r+b\s)=b\ \ \mbox{and}\ \ \eta(a\r+b\s)=a.$ The two functions play the same role on quadrilateral as that of barycentric coordinate on triangles. 

\subsubsection{Triangulations and grids}
Let $\mathcal{Q}_h$ be a regular triangulation of domain $\Omega$, with the cells being convex quadrilaterals; i.e., $\displaystyle\Omega=\cup_{Q\in\mathcal{Q}_h}Q$. Let $\mathcal{N}_h$ denote the set of all the vertices, $\mathcal{N}_h=\mathcal{N}_h^i\cup\mathcal{N}_h^b$, with $\mathcal{N}_h^i$ and $\mathcal{N}_h^b$ consisting of the interior vertices and the boundary vertices, respectively. Similarly, let $\mathcal{E}_h=\mathcal{E}_h^i\bigcup\mathcal{E}_h^b$ denote the set of all the edges, with $\mathcal{E}_h^i$ and $\mathcal{E}_h^b$ consisting of the interior edges and the boundary edges, respectively. For an edge $e$, $\mathbf{n}_e$ is a unit vector normal to $e$, and $\tau_e$ is a unit tangential vector of $e$ such that $\mathbf{n}_e\times\tau_e>0$. On the edge $e$, we use $\llbracket\cdot\rrbracket_e$ for the jump across $e$. 

Denote by $\mathfrak{F}$ the number of cells of the triangulation; denote by $\mathfrak{X}$, $\mathfrak{X}_I$, $\mathfrak{X}_B$ and $\mathfrak{X}_C$ the number of vertices, internal vertices, boundary vertices, and corner vertices, respectively; and denote by $\mathfrak{E}$, $\mathfrak{E}_I$ and $\mathfrak{E}_B$ the number of edges, internal edges, and boundary edges, respectively. Euler's formula states that $\mathfrak{F}+\mathfrak{X}=\mathfrak{E}+1$.

\subsection{An incomplete quadratic finite element on quadrilateral grid}
\label{sec:femecr}

\subsubsection{A finite element on convex quadrilaterals}
The quadrilateral finite element presented below coincides with the one given by Lin-Tobiska-Zhou\cite{Lin.Q;Tobiska.L;Zhou.A2005} on rectangle $Q$, and we call it the quadrilateral Lin-Tobiska-Zhou(QLTZ) element.
\\
{\centering
\fbox{
\begin{minipage}{0.95\textwidth}
The QLTZ element is defined by $(Q,P_Q^{\rm QLTZ},D_Q^{\rm QLTZ})$ with
\begin{enumerate}
\item $Q$ is a convex quadrilateral;
\item $P_Q^{\rm QLTZ}=P_1(Q)+span\{\xi^2,\eta^2\}$;
\item the components of $D_Q^{\rm QLTZ}=\{d_0^{\rm QLTZ}\}_{i=0:4}$ for any $v\in H^1(Q)$ are: 
$$
d_0^{\rm QLTZ}(v)=\fint_Qv\dx,\ \ \mbox{and}\ \ d_i^{\rm QLTZ}(v)=\fint_{e_i}v\dS,\ e_i\ \mbox{the\ edges\ of}\ T,\ i=1:4.
$$
\end{enumerate}
\end{minipage}
}}

The element defined above is unisolvent. Indeed, define
$$
\left\{
\begin{array}{l}
\displaystyle \phi_0=-3(3\xi^2+3\eta^2-2\alpha\xi-2\beta\eta-(4+\alpha^2+\beta^2))/(2\alpha^2+2\beta^2+6), \\
\displaystyle \phi_1=-\frac{3}{4}\xi^2+\frac{\beta-1}{2}\eta+\frac{3+\beta^2}{4}-\frac{\beta^2-\beta+3}{6}\phi_0, \\
\displaystyle \phi_2=-\frac{3}{4}\eta^2+\frac{\alpha-1}{2}\xi+\frac{3+\alpha^2}{4}-\frac{\alpha^2-\alpha+3}{6}\phi_0, \\
\displaystyle \phi_3=-\frac{3}{4}\xi^2+\frac{\beta+1}{2}\eta+\frac{3+\beta^2}{4}-\frac{\beta^2+\beta+3}{6}\phi_0, \\ 
\displaystyle \phi_4=-\frac{3}{4}\eta^2+\frac{\alpha+1}{2}\xi+\frac{3+\alpha^2}{4}-\frac{\alpha^2+\alpha+3}{6}\phi_0,
\end{array}
\right.
$$
then $d_i^{\rm QLTZ}(\phi_j)=\delta_{ij}$, $i,j=0:4$.

Define the interpolation $\Pi_Q^{\rm QLTZ}:H^1(Q)\to P_Q^{\rm QLTZ}$ by $\displaystyle\Pi_Q^{\rm QLTZ}w=\sum_{i=1}^4\fint_{e_i}w\dS\phi_i+\fint_Qw\dx\phi_0.$ Then $\Pi_Q^{\rm QLTZ}$ is well-defined, and $\Pi_Q^{\rm QLTZ}w=w$, if $w\in P_Q^{\rm QLTZ}$.

Let $\mathbf{w}=(w_1,w_2)^T\in (H^1(Q))^2$. We define the interpolator $\undertilde{\Pi}{}_Q^{\rm QLTZ}: (H^1(Q))^2\to (P_Q^{\rm QLTZ})^2$ by two steps: 
\begin{description}
\item[Step 1] Construct $\mathbf{w}^1=(w^1_1,w^1_2)^\top$ by $\displaystyle w^1_i=\sum_{j=1}^4\fint_{E_j}w_i\dS\phi_j$ for $i=1,2$.
\item[Step 2] Find $\mathbf{w}^2=(c_1\phi_0,c_2\phi_0)^\top$ such that
\begin{equation}\label{eq:divpres}
\int_Q\mathrm{div}\mathbf{w}^2q\dx=\int_Q\mathrm{div}(\mathbf{w}-\mathbf{w}^1)q\dx,\quad\forall\,q\in P_1(Q).
\end{equation}
Then define
$$
\undertilde{\Pi}{}_Q^{\rm QLTZ}\mathbf{w}:=\mathbf{w}^1+\mathbf{w}^2.
$$
\end{description}

\begin{lemma}
The interpolator $\undertilde{\Pi}{}_Q^{\rm QLTZ}$ is well-defined. Moreover, $\undertilde{\Pi}{}_Q^{\rm QLTZ}\mathbf{w}=\mathbf{w}$ if $\mathbf{w}\in (P_Q^{\rm QLTZ})^2$, and $\int_Q\mathrm{div}\undertilde{\Pi}{}_Q^{\rm QLTZ}\mathbf{w}q\dx=\int_Q\mathrm{div}\mathbf{w}q\dx$, $\forall\,q\in P_1(Q)$.
\end{lemma}

\begin{proof}
To show the well-definedness of $\undertilde{\Pi}{}_Q^{\rm QLTZ}$, we only have to show that the problem \eqref{eq:divpres} is well-posed. By the definition of $\mathbf{w}^1$, we obtain that $\fint_{e_i}\mathbf{w}-\mathbf{w}^1\dS=\mathbf{0}$ for $i=1:4$. Thus by the property of $\phi_0$, $\int_Q\mathrm{div}(c_1\phi_0,c_2\phi_0)^\top\dx=0=\int_Q\mathrm{div}(\mathbf{w}-\mathbf{w}^1)\dx$ for any $c_1,c_2$. Therefore, we only have to show that the equation \eqref{eq:divpres} admits a unique solution pair $(c_1,c_2)$ for $q$ substituted by $\xi$ and $\eta$. The coefficient matrix of the left hand side of \eqref{eq:divpres} is then
$$
\displaystyle
\left[
\begin{array}{cc}
\displaystyle\int_Q\partial_x\phi_0\xi\dx & \displaystyle\int_Q\partial_y\phi_0 \xi\dx\\
\displaystyle\int_Q\partial_x\phi_0\eta\dx & \displaystyle\int_Q\partial_y\phi_0\eta\dx
\end{array}
\right] = 
\left[
\begin{array}{cc}
\displaystyle\int_Q\partial_\r\phi_0\xi\dx & \displaystyle\int_Q\partial_\s\phi_0 \xi\dx\\
\displaystyle\int_Q\partial_\r\phi_0\eta\dx & \displaystyle\int_Q\partial_\s\phi_0\eta\dx
\end{array}
\right]
\left[
\begin{array}{cc}
\r_x & \s_x \\
\r_y & \s_y
\end{array}
\right]^{-1},
$$
and we only have to check the determinant of the coefficient matrix.

Technically, we construct two tables(Tables \ref{tab:tecasis} and \ref{tab:da}) about the evaluation of some functions firstly. In particular, Table \ref{tab:tecasis} is used in generating Table \ref{tab:da}. For example, we calculate $\int_Q\xi\eta\dx=\int_Q\frac{1}{2}\partial_\s(\xi^2\eta)\dx=\frac{1}{2}\int_{\partial Q}(\xi^2\eta)\s\cdot\mathbf{n}\dS=\frac{1}{2}\sum_{i=1}^4\int_{e_i}(\xi^2\eta)\s\times\tau_i\dS=\frac{4}{3}\alpha\beta\r\times\s$. Here, we have noted that $\partial_\r\xi=\partial_\s\eta=0$ and $\partial_\s\xi=\partial_\r\eta=1$.

\begin{table}[htbp]
\begin{tabular}{c|cccccccc}
 function($u$) & 1 & $\xi$ & $\eta$ & $\xi^2$ & $\eta^2$ & $\xi^3$ & $\eta^3$ & $\xi^2\eta$ \\
\hline $\displaystyle\fint_{e_1}u\dS$ & 1 & 0 & -1 & $\displaystyle \frac{(1-\beta)^2}{3}$ & $\displaystyle 1+\frac{\alpha^2}{3}$&$\displaystyle 0$ & $\displaystyle -1-\alpha^2$ & $\displaystyle -\frac{(1-\beta)^2}{3}$ \\
\hline $\displaystyle\fint_{e_2}u\dS$ & 1 & -1 & 0 & $\displaystyle 1+\frac{\beta^2}{3}$ & $\displaystyle \frac{(1-\alpha)^2}{3}$ & $\displaystyle -1-\beta^2$ & $\displaystyle 0$ & $\displaystyle \frac{2(1-\alpha)\beta}{3}$ \\
\hline $\displaystyle\fint_{e_3}u\dS$ & 1 & 0 & 1 & $\displaystyle \frac{(1+\beta)^2}{3}$ & $\displaystyle 1+\frac{\alpha^2}{3}$ & $\displaystyle 0$ & $\displaystyle 1+\alpha^2$ & $\displaystyle \frac{(1+\beta)^2}{3}$ \\
\hline $\displaystyle\fint_{e_4}u\dS$ & 1 & 1 & 0 & $\displaystyle 1+\frac{\beta^2}{3}$ & $\displaystyle \frac{(1+\alpha)^2}{3}$ & $\displaystyle 1+\beta^2$ & $\displaystyle 0$ & $\displaystyle \frac{2(1+\alpha)\beta}{3}$ \\
\hline
\end{tabular}
\caption{Boundary average of some functions.}\label{tab:tecasis}
\end{table}

\begin{table}[htbp]
\begin{tabular}{c|cccccc} 
function ($u$) & $1$ & $\xi$ & $\eta$ & $\xi^2$ & $\eta^2$ & $\xi\eta$ \\
\hline $\int_Q u\dx$ & $4\r\times\s$ & $\frac{4\alpha}{3}\r\times\s$ & $\frac{4\beta}{3}\r\times\s$ & $\frac{4}{3}(1+\beta^2)\r\times\s$ & $\frac{4}{3}(1+\alpha^2)\r\times\s$ & $\frac{4}{3}\alpha\beta\r\times\s$ \\ 
\hline
\end{tabular}
\caption{Domain average of some functions}\label{tab:da}
\end{table}

As $\phi_0=-3(3\xi^2+3\eta^2-2\alpha\xi-2\beta\eta-(4+\alpha^2+\beta^2))/(2\alpha^2+2\beta^2+6)$, we have
$$
\partial_\r\phi_0=\frac{-9\eta+3\beta}{\alpha^2+\beta^2+3},\quad \ \mbox{and}\quad \partial_\s\phi_0=\frac{-9\xi+3\alpha}{\alpha^2+\beta^2+3}.
$$
Then
$$
\begin{array}{l}
\displaystyle \int_Q\partial_\r\phi_0\xi=\frac{-1}{\alpha^2+\beta^2+3}\int_Q(9\xi\eta-3\beta\xi)\dx=\frac{-8\alpha\beta}{\alpha^2+\beta^2+3}\r\times\s,
\\
\displaystyle \int_Q\partial_\s\phi_0\xi\dx=\frac{-1}{\alpha^2+\beta^2+3}\int_Q(9\xi^2-3\alpha\xi)\dx=\frac{4\alpha^2-12\beta^2-12}{\alpha^2+\beta^2+3}\r\times\s,
\\
\displaystyle \int_Q\partial_\r\phi_0\eta\dx=\frac{-1}{\alpha^2+\beta^2+3}\int_Q(9\eta^2-3\beta\eta)\dx=\frac{-12\alpha^2+4\beta^2-12}{\alpha^2+\beta^2+3}\r\times\s,
\\
\displaystyle \int_Q\partial_\s\phi_0\eta\dx=\frac{-1}{\alpha^2+\beta^2+3}\int_Q(9\xi\eta-3\alpha\eta)\dx=\frac{-8\alpha\beta}{\alpha^2+\beta^2+3}\r\times\s,
\end{array}
$$
and

$$
\mathrm{det}\left[
\begin{array}{cc}
\displaystyle\int_Q\partial_\r\phi_0\xi\dx & \displaystyle\int_Q\partial_\s\phi_0 \xi\dx\\
\displaystyle\int_Q\partial_\r\phi_0\eta\dx & \displaystyle\int_Q\partial_\s\phi_0\eta\dx
\end{array}
\right]=\frac{-48\alpha^4-48\beta^4+96\alpha^2\beta^2+96\alpha^2+96\beta^2+144}{(\alpha^2+\beta^2+3)^2}(\r\times\s)^2.
$$
Therefore, since $0<\alpha,\beta<1$,
\begin{multline*}
\mathrm{det}\left[
\begin{array}{cc}
\displaystyle\int_Q\partial_x\phi_0\xi \dx& \displaystyle\int_Q\partial_y\phi_0 \xi\dx\\
\displaystyle\int_Q\partial_x\phi_0\eta\dx & \displaystyle\int_Q\partial_y\phi_0\eta\dx
\end{array}
\right] =
det\left[
\begin{array}{cc}
\displaystyle\int_Q\partial_\r\phi_0\xi\dx & \displaystyle\int_Q\partial_\s\phi_0 \xi\dx\\
\displaystyle\int_Q\partial_\r\phi_0\eta\dx & \displaystyle\int_Q\partial_\s\phi_0\eta\dx
\end{array}
\right]
det\left[
\begin{array}{cc}
\r_x & \s_x \\
\r_y & \s_y
\end{array}
\right]^{-1} \\ 
=\frac{-48\alpha^4-48\beta^4+96\alpha^2\beta^2+96\alpha^2+96\beta^2+144}{(\alpha^2+\beta^2+3)^2}\r\times\s>0.
\end{multline*}
This proves the well-posedness of \eqref{eq:divpres}, and thus the well-definition of $\undertilde{\Pi}{}_Q^{\rm QLTZ}$.

The remaining follows from the definition of the interpolation. This finishes the proof.
\end{proof}

\subsubsection{A finite element space for $H^1(\Omega)$}

Associated with $H^1(\Omega)$, define a finite element space $V_h^{\rm QLTZ}$ by
$$
V_h^{\rm QLTZ}:=\{w\in L^2(\Omega):w|_Q\in P_Q^{\rm QLTZ}, \fint_e w\dS\ \mbox{is\ continuous\ at}\ e\in\mathcal{E}_h^i\},
$$
and associated with $H^1_0(\Omega)$, define a finite element space $V_{h0}$ by
$$
V_{h0}^{\rm QLTZ}:=\{w_h\in V_h^{\rm QLTZ}:\fint_ew_h\dS=0\ \mbox{at\ }e\in \mathcal{E}_h^b\}.
$$

We define the interpolation operator $\Pi_h^{\rm QLTZ}:H^1(\Omega)\to V_h^{\rm QLTZ}$ by 
$$
\Pi_h^{\rm QLTZ}w\in V_h^{\rm QLTZ},\ \ (\Pi_h^{\rm QLTZ}w)|_Q=\Pi_Q^{\rm QLTZ}(w|_Q), \ \mbox{for}\,w\in H^1(\Omega).
$$
The well-definedness of $\Pi_h^{\rm QLTZ}$ is evident. Moreover, $\Pi_h^{\rm QLTZ}w\in V_{h0}^{\rm QLTZ}$, if $w\in H^1_0(\Omega)$.

Associated with $(H^1(\Omega))^2$ (and $(H^1_0(\Omega))^2$), we define the finite element space $\mathbf{V}_h^{\rm QLTZ}:=(V_h^{\rm QLTZ})^2$ (and $\mathbf{V}_{h0}^{\rm QLTZ}:=(V_{h0}^{\rm QLTZ})^2$, respectively). Define the interpolation operator $\undertilde{\Pi}{}_h^{\rm QLTZ}:(H^1(\Omega))^2\to\mathbf{V}_h^{\rm QLTZ}$ by
$$
\undertilde{\Pi}{}_h^{\rm QLTZ}\mathbf{w}\in \mathbf{V}_h^{\rm QLTZ},\ (\undertilde{\Pi}{}_h^{\rm QLTZ}\mathbf{w})|_Q=\undertilde{\Pi}{}_Q^{\rm QLTZ}(\mathbf{w}|_Q),\ \mbox{for}\,\mathbf{w}\in (H^1(\Omega))^2.
$$
Again, $\undertilde{\Pi}{}_h^{\rm QLTZ}$ is well-defined, and $\undertilde{\Pi}{}_h^{\rm QLTZ}\mathbf{w}\in\mathbf{V}_{h0}^{\rm QLTZ}$, if $\mathbf{w}\in (H^1_0(\Omega))^2$. 

Evidently, $\Pi_h^{\rm QLTZ}w_h=w_h$ for $w_h\in V_h^{\rm QLTZ}$ and $\undertilde{\Pi}{}_h^{\rm QLTZ}\mathbf{w}_h=\mathbf{w}_h$ for $\mathbf{w}_h\in \mathbf{V}_h^{\rm QLTZ}$. Thus, since $P_1(Q)\subset P_Q^{\rm QLTZ}$, by standard technique\cite{Lin.Q;Tobiska.L;Zhou.A2005,Shi.Z;Wang.M2013mono,Wang.M;Shi.Z;Xu.J2007,Park.C;Sheen.D2013}, we have the lemma below.
{
\begin{lemma}\label{lem:interr}
Let $\{\mathcal{Q}_h\}$ be a regular family of convex quadrilateral triangulations of $\Omega$. 
\begin{enumerate}
\item There exists a constant $C$, such that it holds for $w\in H^s(\Omega)$, $s=1,2$, that 
$$
|w-\Pi_h^{\rm QLTZ}w|_{m,h}\leqslant Ch^{s-m}|w|_{s,\Omega},\ 0\leqslant m\leqslant s.
$$
\item There exists a constant $C$, such that it holds for $\mathbf{w}\in (H^s(\Omega))^2$, $s=1,2$, that 
$$
|\mathbf{w}-\undertilde{\Pi}{}_h^{\rm QLTZ}\mathbf{w}|_{m,h}\leqslant Ch^{s-m}|\mathbf{w}|_{s,\Omega},\ 0\leqslant m\leqslant s.
$$
\end{enumerate}
\end{lemma}
}

\subsubsection{Application to second order elliptic problem}
We consider the variational problem: find $u\in H^1_0(\Omega)$, such that
\begin{equation}\label{eq:poissonvp}
(\nabla u,\nabla v)=(f,v),\ \forall\,v\in H^1_0(\Omega).
\end{equation}
Then $V_h^{\rm QLTZ}$ is a consistent finite element space. The finite element problem is to find $u_h\in V_{h0}^{\rm QLTZ}$, such that
\begin{equation}\label{eq:poissonfep}
(\nabla_hu_h,\nabla_hv_h)=(f,v_h),\ \forall\,v_h\in V_{h0}^{\rm QLTZ}.
\end{equation}
Since the edge average of $w_h\in V_{h0}^{\rm QLTZ}$ is continuous across internal edges, by the standard technique, we have the error estimate below.
\begin{theorem}
Let the assumptions of Lemma \ref{lem:interr} hold. Let $u$ and $u_h$ be the solutions of \eqref{eq:poissonvp} and \eqref{eq:poissonfep}, respectively. 
\begin{enumerate}
\item If $u\in H^2(\Omega)\cap H^1_0(\Omega)$, then $\|u-u_h\|_{1,h}\lesssim h|u|_{2,\Omega}$.
\item If $\Omega$ is convex and $f\in L^2(\Omega)$, then $\|u-u_h\|_{0,\Omega}\lesssim h^2\|f\|_{0,\Omega}$.
\end{enumerate}
\end{theorem}
From this point onwards, $\lesssim$, $\gtrsim$, and $\cequiv$ respectively denote $\leqslant$, $\geqslant$, and $=$ up to a constant. The hidden constants depend on the domain. And, when triangulation is involved, they also depend on the shape-regularity of the triangulation, but they do not depend on $h$ or any other mesh parameter.

\subsection{A stable finite element pair for Stokes problem}
For Stokes problem, $\mathbf{V}_h^{\rm QLTZ}$ provides a consistent finite element space for the velocity field; we a in lack of a space $\mathring{W}_h^\q$ for the pressure. To satisfy \textbf{SC 1} and \textbf{SC 2} at the same time, we need $\mathrm{div}_h\mathbf{V}_{h0}^{\rm QLTZ}= \mathring{W}_h^\q$. Here we use the space of piecewise linear polynomials for the pressure field. Define $W_h^\q:=\{q_h\in L^2(\Omega):q_h|_Q\in P_1(Q),\ \forall\,Q\in\mathcal{Q}_h\}$ and $\mathring{W}_h^\q=W_h^\q\cap L^2_0(\Omega)$. The finite element problem is to find $(\mathbf{u}_h,p_h)\in\mathbf{V}_{h0}^{\rm QLTZ}\times \mathring{W}_h^\q$, such that
\begin{equation}\label{eq:stokesfep}
\left\{
\begin{array}{rll}
(\nabla_h\mathbf{u}_h,\nabla_h\mathbf{v}_h)+(\mathrm{div}_h\mathbf{v}_h,p_h)&=(\mathbf{f},\mathbf{v}_h) & \forall\,\mathbf{v}_h\in \mathbf{V}_{h0}^{\rm QLTZ}; \\ 
(\mathrm{div}_h\mathbf{u}_h,q_h)&=0 & \forall\,q_h\in \mathring{W}_h^\q.
\end{array}
\right.
\end{equation}

\begin{lemma}\label{lem:discinfsup}
The inf-sup condition holds for $\mathbf{V}_{h0}^{\rm QLTZ}\times \mathring{W}_h^\q$ that
\begin{equation}\label{eq:discinfsup}
\inf_{q_h\in \mathring{W}_h^\q}\sup_{\mathbf{v}_h\in \mathbf{V}_{h0}^{\rm QLTZ}}\frac{(\mathrm{div}_h \mathbf{v}_h,q_h)}{\|\mathbf{v}_h\|_{1,h}\|q_h\|_{0,\Omega}}\geqslant C \mbox{(independent of $h$)}.
\end{equation}
\end{lemma}
\begin{proof}
Given $q_h\in \mathring{W}_h^\q\subset L^2_0(\Omega)$, there exists $\mathbf{w}\in (H^1_0(\Omega))^2$, such that $q_h=\mathrm{div}\mathbf{w}$, and $\|\mathbf{w}\|_{1,\Omega}\leqslant C_1\|\mathrm{div}\mathbf{w}\|_{0,\Omega}$\cite{Girault.V;Raviart.P}. Here $C_1$ is a generic constant depending on the domain only. Define $\mathbf{w}_h:=\undertilde{\Pi}{}_h^{\rm QLTZ}\mathbf{w}$, and then $\int_Q \mathrm{div}\mathbf{w}_hq\dx=\int_Q\mathrm{div}\mathbf{w}q\dx$ for $q\in P_1(Q)$. Since $\mathrm{div}(\mathbf{w}_h|_T)\in P_1(Q)$ and $(\mathrm{div}\mathbf{w})|_Q=q_h|_Q\in P_1(Q)$, this implies $\mathrm{div}(\mathbf{w}_h|_Q)=(\mathrm{div}\mathbf{w})|_Q$ and further $\mathrm{div}_h\mathbf{w}_h=\mathrm{div}\mathbf{w}$. Therefore,
$$
\sup_{\mathbf{v}_h\in\mathbf{V}_{h0}^{\rm QLTZ}}\frac{(\mathrm{div}_h\mathbf{v}_h,q_h)}{\|\mathbf{v}_h\|_{1,h}\|q_h\|_{0,\Omega}}\geqslant \frac{(\mathrm{div}_h\mathbf{w}_h,q_h)}{\|\mathbf{w}_h\|_{1,h}\|q_h\|_{0,\Omega}}\geqslant C' \frac{(div\mathbf{w},q_h)}{\|\mathbf{w}\|_{1,\Omega}\|q_h\|_{0,\Omega}}\geqslant C.
$$
The last second inequality follows from Lemma \ref{lem:interr}. This finishes the proof.
\end{proof}
\begin{remark}\label{rem:discsur}
By the proof of Lemma \ref{lem:discinfsup}, $\mathrm{div}_h\mathbf{V}_{h0}^{\rm QLTZ}=\mathring{W}_h^\q$. Simultaneously, $\mathrm{curl}_h\mathbf{V}_{h0}^{\rm QLTZ}=\mathring{W}_h^\q$.
\end{remark}

Again, since the edge averages of $\mathbf{w}_h\in \mathbf{V}_{h0}^{\rm QLTZ}$ are continuous across internal edges, by the standard technique, the theorem below follows from Lemmas \ref{lem:clm} and \ref{lem:discinfsup}.

\begin{theorem}
Let $(\mathbf{u},p)$ and $(\mathbf{u}_h,p_h)$ be the solutions of \eqref{eq:stokesvp} and \eqref{eq:stokesfep}, respectively. Then $\mathrm{div}_h\mathbf{u}_h=0$. Moreover,
\begin{enumerate}
\item If $\mathbf{u}\in (H^2(\Omega)\cap H^1_0(\Omega))^2$ and $p\in H^1(\Omega)\cap L^2_0(\Omega)$, then 
$$
\|\mathbf{u}-\mathbf{u}_h\|_{1,h}+\|p-p_h\|_{0,\Omega}\lesssim h(|\mathbf{u}|_{2,\Omega}+|p|_{1,\Omega});
$$
\item If $\Omega$ is convex, then $\|\mathbf{u}-\mathbf{u}_h\|_{0,\Omega}\lesssim h^2\|\mathbf{f}\|_{0,\Omega}$.
\end{enumerate}
\end{theorem}

%

\subsection{Discrete Stokes complex}

\subsubsection{A Morley element on convex quadrilateral grid}
This quadrilateral Morley element is given by Park-Sheen \cite{Park.C;Sheen.D2013}. \\

{\centering
\fbox{
\begin{minipage}{0.95\textwidth}
The quadrilateral Morley element is defined by $(Q,P_Q^M,D_Q^M)$ with
\begin{enumerate}
\item $Q$ is a convex quadrilateral;
\item $P_Q^M=P_2(Q)+span\{\xi^3,\eta^3\}$;
\item the components of $D_Q^M=\{d_i^M,d_{i+4}^M\}_{i=1:4}$ for any $v\in H^2(Q)$ are: 
$$
d^M_i(v)=v(a_i),\ a_i\ \mbox{the\ vertices\ of}\ T;\ \ d^M_{i+4}(v)=\fint_{e_i}\partial_{\mathbf{n}_{e_i}}v\dS,\ e_i\ \mbox{the\ edges\ of}\ T.
$$
\end{enumerate}
\end{minipage}
}}

Given a regular convex quadrilateral triangulation of $\Omega$, define the Morley element space $M_h^\q$ as
\begin{multline*}
\qquad M_h^\q:=\{w_h\in L^2(\Omega):w_h|_Q\in P_Q^M,\ w_h(a)\ \mbox{is\ continuous\ at}\ a\in\mathcal{N}_h,\  \\ 
\fint_e\partial_{\mathbf{n}_e}w_h\dS\ \mbox{is\ continuous\ across}\ e\in\mathcal{E}_h^i\}.\qquad
\end{multline*}
And, associated with $H^2_0(\Omega)$, define 
$$
M_{h0}^\q:=\{w_h\in M_h:w_h(a)\ \mbox{vanishes\ at}\ a\in\mathcal{N}_h^b,\ \fint_e\partial_{\mathbf{n}_e}w_h\dS\ \mbox{vanishes\ at}\ e\in\mathcal{E}_h^b\}.
$$

The Morley element provides consistent approximation of fourth-order problems on quadrilateral grids. Let us consider the model problem: find $u\in H^2_0(\Omega)$, such that
\begin{equation}\label{eq:bihvp}
(\nabla^2u,\nabla^2v)=(f,v)\quad \forall\,v\in H^2_0(\Omega).
\end{equation}
The finite element problem is to find $u_h\in M_{h0}^\q$, such that
\begin{equation}\label{eq:bihfep}
(\nabla_h^2u_h,\nabla_h^2v_h)=(f,v_h)\quad\forall\,v_h\in M_{h0}^\q.
\end{equation}
\begin{lemma}\label{lem:parksheen}\cite{Park.C;Sheen.D2013}
Let $u$ and $u_h$ be the solution of \eqref{eq:bihvp} and \eqref{eq:bihfep}, respectively.
\begin{enumerate}
\item Assume $u\in H^3(\Omega)\cap H^2_0(\Omega)$, then $|u-u_h|_{2,h}\lesssim h(|u|_{3,\Omega}+\|f\|_{0,\Omega})$.
\item If further $\Omega$ is a convex polygon, then $|u-u_h|_{1,h}\lesssim h^2(|u|_{3,\Omega}+\|f\|_{0,\Omega})$.
\end{enumerate}
\end{lemma}

\subsubsection{A discrete Stokes complex}

The lemma below plays a fundamental role in the construction of the discrete Stokes complex.
\begin{lemma}\label{lem:nullembd}
$\mathbf{curl}_h M_{h0}^\q=\widetilde{\mathbf{V}}_{h0}^{\rm QLTZ}:=\{\mathbf{w}_h\in\mathbf{V}_{h0}^{\rm QLTZ}:\mathrm{div}_h\mathbf{w}_h=0\}$.
\end{lemma}
\begin{proof}
It is obvious that $\textbf{curl}_h M_{h0}^\q\subset \widetilde{\mathbf{V}}_{h0}^{\rm QLTZ}$. To prove the other direction, we only have to show that the dimension of the two spaces are the same.  By Remark \ref{rem:discsur}, 
\begin{multline*}
\dim(\widetilde{\mathbf{V}}_{h0}^{\rm QLTZ})=\dim(\mathbf{V}_{h0}^{\rm QLTZ})-\dim(\mathrm{div}_h\mathbf{V}_{h0}^{\rm QLTZ})=\dim(\mathbf{V}_{h0}^{\rm QLTZ})-\dim(\mathring{W}_h^\q) \\ 
=2(\mathfrak{F}+\mathfrak{E}_I)-(3\mathfrak{F}-1)=\mathfrak{E}_I+\mathfrak{X}_I=\dim(M_{h0}^\q)=\dim(\mathbf{curl}_hM_{h0}^\q).\qquad
\end{multline*}
This finishes the proof.
\end{proof}

Define $\Pi_h^{\q M}:H^2(\Omega)\to M_h^\q$ by $\Pi_h^{\q M}\varphi\in M_h^\q$ such that
$$
\Pi_h^{\q M}\varphi(a)=\varphi(a),\ \forall\,a\in \mathcal{N}_h, \ \mbox{and} \fint_e\partial_{\mathbf{n}_e}\Pi_h^{\q M}\varphi\dS=\fint_e\partial_{\mathbf{n}_e}\varphi\dS, \ \forall\,e\in\mathcal{E}_h,\ \varphi\in H^2(\Omega).
$$
Define $\Pi_h^0$ the $L^2$-projection to $W_h^\q$. Summing all discussions above, we obtain a main result of the paper as below.

\begin{theorem}\label{thm:disStoSim1}
The discrete Stokes complex holds as below:
\begin{equation}
\begin{array}{ccccccccc}
0 & \longrightarrow & M_{h0}^\q & \xrightarrow{\bs{\mrm{curl}}_h} & \mathbf{V}_{h0}^{\rm QLTZ} & \xrightarrow{\mrm{div}_h} & \mathring{W}_h^\q  & \longrightarrow & 0.
\end{array}
\end{equation}
%
Moreover,
\begin{equation}\label{eq:commdia}
\bs{\mrm{curl}}_h \Pi_{h}^{\q\bs{\mrm{M}}}=\undertilde{\Pi}{}_{h}^{\rm QLTZ}\bs{\mrm{curl}}\ \mbox{on}\ H^2_0(\Omega),\quad\mbox{and}\quad \mrm{div}_h\undertilde{\Pi}{}_{h}^{\rm QLTZ}=\Pi_h^0\mrm{div}\ \mbox{on}\ (H^1_0(\Omega))^2.
\end{equation}
\end{theorem}
\begin{proof}
The discrete Stokes complex follows from Lemma \ref{lem:nullembd} and Remark \ref{rem:discsur}. We only have to prove the commutativity \eqref{eq:commdia}.

Given $\varphi\in H^2_0(\Omega)$, by the definition of $\Pi_h^{\q M}$, we have for $e\in \mathcal{E}_h$ that
\begin{multline*}
\fint_e\mathbf{curl}_h\Pi_h^{\q M}\varphi\dS = \fint_e\mathbf{curl}_h\Pi_h^{\q M}\varphi\cdot\mathbf{n}_e\dS\mathbf{n}_e + \fint_e\mathbf{curl}_h\Pi_h^{\q M}\varphi\cdot\mathbf{\tau}_e\dS\mathbf{\tau}_e \\ 
= \fint_e\partial_{\tau_e}\Pi_h^{\q M}\varphi\dS\mathbf{n}_e+\fint_e\partial_{\mathbf{n}_e}\Pi_h^{\q M}\varphi\dS\mathbf{\tau}_e=(\Pi_h^{\q M}\varphi(e_L)-\Pi_h^M\varphi(e_R))\mathbf{n}_e + \fint_e\partial_{\mathbf{n}_e}\Pi_h^{\q M}\varphi\dS \mathbf{\tau}_e \\ 
= (\varphi(e_L)-\varphi(e_R))\mathbf{n}_e + \fint_e\partial_{\mathbf{n}_e}\varphi\dS\mathbf{\tau}_e = \fint_e\partial_{\tau_e}\varphi\dS\mathbf{n}_e+\fint_e\partial_{\mathbf{n}_e}\varphi\dS \mathbf{\tau}_e = \fint_e\mathbf{curl}\,\varphi\dS. 
\end{multline*}
Note that $\mathrm{div}_h\mathbf{curl}_h\Pi_h^{\q M}\varphi=0=\mathrm{div}\undertilde{\Pi}{}_h^{\rm QLTZ}\mathbf{curl}\,\varphi$. Therefore, $\mathbf{curl}_h\Pi_h^{\q M}\varphi=\undertilde{\Pi}{}_h^{\rm QLTZ}\mathbf{curl}\,\varphi$. Similarly we can prove for $\mathbf{v}\in (H^1_0(\Omega))^2$ that    $\mrm{div}_h\undertilde{\Pi}{}_{h}^{\rm QLTZ}\mathbf{v}=\Pi_h^0\mrm{div}\mathbf{v}$. This finishes the proof.
\end{proof}

Theorem \ref{thm:disStoSim1} can also be written as this exact sequence and commutative diagram:
\begin{equation}
\begin{array}{ccccccccc}
0 & ~~~\longrightarrow~~~ & H^2_0(\Omega) & ~~~\xrightarrow{\bs{\mrm{curl}}}~~~ & (H^1_0(\Omega))^2 & ~~~\xrightarrow{\mrm{div}}~~~ & L^2_0(\Omega)  & ~~~\rightarrow~~~ & 0  \\
 & & \downarrow \Pi_{h}^{\q\bs{\mrm{M}}} & & \downarrow \Pi_{h}^{\rm QLTZ} & & \downarrow \Pi_{h}^{0} & & \\
0 & \rightarrow & M_{h0}^\q & \xrightarrow{\bs{\mrm{curl}}_h} & \mathbf{V}_{h0}^{\rm QLTZ} & \xrightarrow{\mrm{div}_h} & \mathring{W}_h^\q  & \rightarrow & 0.
\end{array}
\end{equation}


%
%
%
%
%

\section{Finite elements and discrete Stokes complex on a mixed grid}
\label{sec:mixgrid}

In this section, we generalise the results in Section \ref{sec:stap} from quadrilateral grids to the mixed grid that consists  of both triangular and quadrilateral cells. The technical issues are the same as that in Section \ref{sec:stap}, and we list the main results and omit the details.


\subsection{Mixed triangulation with triangular and quadrilateral cells}

Let $\mathcal{T}_h$ be a shape-regular triangulation of domain $\Omega$, with the cells being triangles or convex quadrilaterals. Again, let $\mathcal{N}_h$ denote the set of all the vertices, $\mathcal{N}_h=\mathcal{N}_h^i\cup\mathcal{N}_h^b$, with $\mathcal{N}_h^i$ and $\mathcal{N}_h^b$ consisting of the interior vertices and the boundary vertices, respectively. Similarly, let $\mathcal{E}_h=\mathcal{E}_h^i\bigcup\mathcal{E}_h^b$ denote the set of all the edges, with $\mathcal{E}_h^i$ and $\mathcal{E}_h^b$ consisting of the interior edges and the boundary edges, respectively. For an edge $e$, $\mathbf{n}_e$ is a unit vector normal to $e$, and $\tau_e$ is a unit tangential vector of $e$ such that $\mathbf{n}_e\times\tau_e>0$. On the edge $e$, we use $\llbracket\cdot\rrbracket_e$ for the jump across $e$. 

Again, denote by $\mathfrak{F}$ the number of cells of the triangulation, denote by $\mathfrak{X}$, $\mathfrak{X}_I$, $\mathfrak{X}_B$ and $\mathfrak{X}_C$ the number of vertices, internal vertices, boundary vertices, and corner vertices, respectively, and denote by $\mathfrak{E}$, $\mathfrak{E}_I$ and $\mathfrak{E}_B$ the number of edges, internal edges, and boundary edges, respectively. Euler's formula states that $\mathfrak{F}+\mathfrak{X}=\mathfrak{E}+1$. In the remaining of this section, we use $Q$ to denote a quadrilateral cell, and $T$ for a triangular cell. This will not bring ambiguity according to the context. Denote by $\# Q$ and $\# T$ the number of quadrilateral and triangular cells, respectively.

\subsection{Finite element spaces on a mixed grid}

Associated with the triangulation, we define several finite element spaces for the stream function, the velocity and the pressure, respectively. Associated with the stream function, define
\begin{multline*}
 M_h^\m:=\{w_h\in L^2(\Omega):w_h|_Q\in P_Q^M,\ w_h|_T\in P_2(T),\\ 
\quad w_h(a)\ \mbox{is\ continuous\ at}\ a\in\mathcal{N}_h,\ \fint_e\partial_{\mathbf{n}_e}w_h\dS\ \mbox{is\ continuous\ on}\ e\in\mathcal{E}_h^i\}.\qquad
\end{multline*}
And, associated with $H^2_0(\Omega)$, define 
$$
M_{h0}^\m:=\{w_h\in M_h:w_h(a)=0\ \mbox{at}\ a\in\mathcal{N}_h^b,\ \fint_e\partial_{\mathbf{n}_e}w_h\dS=0\ \mbox{at}\ e\in\mathcal{E}_h^b\}.
$$
Define the interpolation $\Pi_h^{M,\m}:H^2(\Omega)\to M_h^\m$ by $\Pi_h^{M,\m}\varphi\in M_h^\m$, by
$$
\Pi_h^{M,\m}\varphi(a)=\varphi(a),\ \forall\,a\in \mathcal{N}_h, \ \mbox{and} \fint_e\partial_{\mathbf{n}_e}\Pi_h^{M,\m}\varphi\dS=\fint_e\partial_{\mathbf{n}_e}\varphi\dS, \ \forall\,e\in\mathcal{E}_h,
$$
for $\varphi\in H^2(\Omega)$. Then $\Pi_h^{M,\m}$ is well-defined, and $\Pi_h^{M,\m} H^2_0(\Omega)= M_{h0}^\m$.

Define associated with $H^1(\Omega)$
$$
V_h^\m:=\{w\in L^2(\Omega):w|_Q\in P_Q,\ w|_T\in P_1(T),\ \fint_e w\dS\ \mbox{is\ continuous\ on}\ e\in\mathcal{E}_h^i\},
$$
and associated with $H^1_0(\Omega)$,
$$
V_{h0}^\m:=\{w_h\in V_h:\fint_ew_h\dS=0,\ \mbox{for\ }e\in \mathcal{E}_h^b\}.
$$
Define $\mathbf{V}_h^\m=(V_h^\m)^2$ and $\mathbf{V}_{h0}^\m=(V_{h0}^\m)^2$. Define $\undertilde{\Pi}{}_h^{V,\m}:(H^1(\Omega))^2\to \mathbf{V}_h^\m$ by $\undertilde{\Pi}{}_h^{V,\m}\mathbf{w}\in \mathbf{V}_h^\m$ such that
$$
\fint_e\undertilde{\Pi}{}_h^{V,\m}\mathbf{w}\dS=\fint_e\mathbf{w}\dS,\ \forall\,e\in\mathcal{E}_h,\ \mbox{and}\,(\undertilde{\Pi}{}_h^{V,\m}\mathbf{w})|_Q=\undertilde{\Pi}{}_Q^{\rm QLTZ}\mathbf{w}
$$
for $\mathbf{w}\in (H^1(\Omega))^2$. Then $\undertilde{\Pi}{}_h^{V,\m}$ is well-defined, and $\undertilde{\Pi}{}_h^\m(H^1_0(\Omega))^2=\mathbf{V}_{h0}^\m$.

Define associated with $L^2(\Omega)$
$$
W_h^\m:=\{q_h\in L^2(\Omega):q_h|_Q\in P_1(Q),\ q_h|_T\in P_0(T)\},
$$
and associated with $L^2_0(\Omega)$, $\mathring{W}_h^\m:=W_h^\m\cap L^2_0(\Omega).$ Define $\Pi_h^{0,\m}$ the $L^2$-projection to $W_h^\m$. Then $\Pi_h^{0,\m} L^2_0(\Omega)=\mathring{W}_h^\m$.

Note that the finite element functions in the spaces $(M_h^\m,V_h^\m\times W_h^\m)$ coincide with the original Morley element functions and the nonconforming $P_1^2-P_0$ element functions if restricted on triangular cells, and with quadrilateral Morley element functions and the newly-developed Stokes element functions if restricted on quadrilateral cells.  The operators $\Pi_h^{M,\m}$, $\undertilde{\Pi}{}_h^{V,\m}$ and $\Pi_h^{0,\m}$ are all defined cell by cell. Roughly speaking, they are ``sum direct" of the interpolations with respect to the finite elements defined on individual cells.

\subsection{Finite element schemes for boundary value problems}

Since the nodal interpolation operators are locally defined, and sufficient weak continuity conditions have been imposed across the interface, these finite elements defined previously provide convergent finite element schemes for the specific boundary value problems. 

\paragraph{\textbf{Fourth order problem}} We consider the finite element problem: find $u_h\in M_{h0}^\m$, such that
\begin{equation}\label{eq:bihfepmix}
(\nabla_h^2u_h,\nabla_h^2v_h)=(f,v_h)\quad\forall\,v_h\in M_{h0}^\m.
\end{equation}
Since the average of the gradient of the $M_{h0}^\m$ functions at the internal edges are continuous, by standard technique\cite{Shi.Z1990,Wang.M;Shi.Z;Xu.J2007,Park.C;Sheen.D2013}, we have the convergence result below.
\begin{lemma}\label{lem:morleymix}
Let $u$ and $u_h$ be the solution of \eqref{eq:bihvp} and \eqref{eq:bihfepmix}, respectively.
\begin{enumerate}
\item Assume $u\in H^3(\Omega)\cap H^2_0(\Omega)$, then $|u-u_h|_{2,h}\lesssim h(|u|_{3,\Omega}+\|f\|_{0,\Omega})$.
\item If further $\Omega$ is a convex polygon, then $|u-u_h|_{1,h}\lesssim h^2(|u|_{3,\Omega}+\|f\|_{0,\Omega})$.
\end{enumerate}
\end{lemma}

\paragraph{\textbf{Stokes problem}} Now we consider the finite element problem: find $(\mathbf{u}_h,p_h)\in\mathbf{V}_{h0}^\m\times \mathring{W}_h^\m$, such that
\begin{equation}\label{eq:stokesfepmix}
\left\{
\begin{array}{rll}
(\nabla_h\mathbf{u}_h,\nabla\mathbf{v}_h)+(\mathrm{div}_h\mathbf{v}_h,p_h)&=(\mathbf{f},\mathbf{v}_h) & \forall\,\mathbf{v}_h\in \mathbf{V}_{h0}^\m; \\ 
(\mathrm{div}_h\mathbf{u}_h,q_h)&=0 & \forall\,q_h\in \mathring{W}_h^\m.
\end{array}
\right.
\end{equation}
By the same technique as the proof of Lemma \ref{lem:discinfsup}, we obtain the lemma below.
\begin{lemma}
The inf-sup condition holds for $\mathbf{V}_{h0}^\m\times \mathring{W}_h^\m$ that
\begin{equation}\label{eq:discinfsupmix}
\inf_{q_h\in \mathring{W}_h^\m}\sup_{\mathbf{v}_h\in \mathbf{V}_{h0}^\m}\frac{(\mathrm{div}_h \mathbf{v}_h,q_h)}{\|\mathbf{v}_h\|_{1,h}\|q_h\|_{0,\Omega}}\geqslant C \mbox{(independent of $h$)}.
\end{equation}
Moreover, $\mathrm{div}_h\mathbf{V}_{h0}^\m=\mathring{W}_h^\m$, and $\mathbf{curl}_h\mathbf{V}_{h0}^\m=\mathring{W}_h^\m$.
\end{lemma}

Again, since the edge average of $\mathbf{w}_h\in \mathbf{V}_{h0}^\m$ is continuous across internal edges, by the standard technique, we obtain the estimate below.
\begin{lemma}
Let $(\mathbf{u},p)$ and $(\mathbf{u}_h,p_h)$ be the solutions of \eqref{eq:stokesvp} and \eqref{eq:stokesfepmix}, respectively. Then $\mathrm{div}_h\mathbf{u}_h=0$. Moreover,
\begin{enumerate}
\item If $\mathbf{u}\in (H^2(\Omega)\cap H^1_0(\Omega))^2$ and $p\in H^1(\Omega)\cap L^2_0(\Omega)$, then 
$$
\|\textbf{u}-\textbf{u}_h\|_{1,h}+\|p-p_h\|_{0,\Omega}\lesssim h(|\textbf{u}|_{2,\Omega}+|p|_{1,\Omega});
$$
\item If $\Omega$ is convex, then $\|\textbf{u}-\textbf{u}_h\|_{0,\Omega}\lesssim h^2\|\mathbf{f}\|_{0,\Omega}$.
\end{enumerate}
\end{lemma}

\subsection{Discrete Stokes complex on a mixed grid}
Direct calculation leads to that 
$$
\dim(\mathbf{V}_{h0}^\m)-\dim(\mathring{W}_{h0})=2(\mathfrak{E}_I+\# Q)-(3\#Q+\# T-1)=\mathfrak{E}_I+\mathfrak{X}_I=\dim(M_{h0}^\m).
$$
Therefore, we can apply the technique of the proof of Lemma \ref{lem:nullembd} and of the proof of Theorem \ref{thm:disStoSim1} to obtain the theorem below.

\begin{theorem}\label{thm:disStoSimmix}
The discrete Stokes complex holds as below:
$$
\begin{array}{ccccccccc}
0 & \longrightarrow & M_{h0}^\m & \xrightarrow{\bs{\mrm{curl}}_h} & \mathbf{V}_{h0}^\m & \xrightarrow{\mrm{div}_h} & \mathring{W}_h^\m  & \longrightarrow & 0.
\end{array}
$$
Moreover,
$$
\bs{\mrm{curl}}_h \Pi_{h}^{\bs{\mrm{M}},\m}=\undertilde{\Pi}{}_{h}^{\mrm{V},\m}\bs{\mrm{curl}}\ \mbox{on}\ H^2_0(\Omega),\quad\mbox{and}\quad \mrm{div}_h\undertilde{\Pi}{}_{h}^{\mrm{V},\m}=\Pi_h^{0,\m}\mrm{div}\ \mbox{on}\ (H^1_0(\Omega))^2.
$$
\end{theorem}

Again, the theorem can be written as the exact sequence and the commutative diagram below:
\begin{equation}
\begin{array}{ccccccccc}
0 & ~~~\longrightarrow~~~ & H^2_0(\Omega) & ~~~\xrightarrow{\bs{\mrm{curl}}}~~~ & (H^1_0(\Omega))^2 & ~~~\xrightarrow{\mrm{div}}~~~ & L^2_0(\Omega)  & ~~~\longrightarrow~~~ & 0  \\
 & & \downarrow \Pi_{h}^{\bs{\mrm{M}},\m} & & \downarrow \undertilde{\Pi}{}_{h}^{\mrm{V},\m} & & \downarrow \Pi_{h}^{0,\m} & & \\
0 & \longrightarrow & M_{h0}^\m & \xrightarrow{\bs{\mrm{curl}}_h} & \mathbf{V}_{h0}^\m & \xrightarrow{\mrm{div}_h} & \mathring{W}_h^\m & \longrightarrow & 0.
\end{array}
\end{equation}

\section{Concluding remarks}
\label{sec:con}
{
In this paper, we construct stable finite element pairs that satisfy the stability conditions both \textbf{SC 1} and \textbf{SC 2} on grids that admit triangular and general convex quadrilateral cells, namely, the pair satisfies the inf-sup stability condition, and the restriction of the solution to an element is exactly divergence-free and the scheme can be seen as a mass conservative one. Different from most existing finite element pairs on quadrilateral grids in the literature, the construction of the newly-developed quadrilateral finite element spaces does not rely on a rectangle reference cell, and the finite element spaces thus consist of piecewise polynomials only. Discrete Stokes complexes are constructed associatedly.
}

As the constraint of divergence-free is imposed piecewisely, this finite element would have potential applications for the parametric related problems \cite{Mardal.K;Tai.X;Winther.R2002,Xie.X;Xu.J;Xue.G2008}. The exact sequence property provides a precise description of the kernel space involved in the Stokes problem, and will also help to design preconditioners and solvers for the resulting linear systems \cite{Feng.C;Xu.J;Zhang.S2013,Hiptmair.R;Xu.J2007,Xie.X;Xu.J;Xue.G2008,Mardal.K;Schoberl.J;Winther.R2012}. These will be discussed in future works.

As the finite elements constructed in this present paper fall into the category of the nonconforming type, an issue is that the uniform Korn's inequality would fail\cite{Knoblock.P;Tobiska.L2005}. This issue will be discussed in future works. We also remark here that the piecewise mass conservation property of the finite element pair makes it potentially one fit for the elasticity problem, and we refer to \cite{Arnold.D1994} for a relevant discussion.


\end{document}